\documentclass[12pt]{amsart}
\usepackage{amssymb, a4wide, mathdots, url,hyperref}
\usepackage[usenames]{color}

\newcommand{\abs}[1]{\left|{#1}\right|}

\newcommand{\wt}{\widetilde}
\newcommand{\GL}{\operatorname{GL}}

\newcommand{\Ad}{\operatorname{Ad}}

\newcommand{\Nu}{\mathcal{V}}

\newcommand{\sprod}[2]{\left\langle{#1},{#2}\right\rangle}

\renewcommand{\Re}{\operatorname{Re}}

\newcommand{\N}{\mathbb{N}}

\newcommand{\bs}{\backslash}

\newcommand{\diag}{\operatorname{diag}}

\newcommand{\Hom}{\operatorname{Hom}}

\newcommand{\K}{\mathcal{K}}
\newcommand{\Wh}{\mathcal{W}}

\newcommand{\gk}{{\operatorname{GK}}}

\newtheorem{theorem}{Theorem}[section]

\theoremstyle{remark}
\newtheorem{remark}[theorem]{Remark}

\title{A note on gamma factors for pairs}

\author{Omer Offen}
\address{Department of Mathematics, Brandeis University, Waltham MA 02453}
\email{offen@brandeis.edu}
\date{\today}

\begin{document}\maketitle
\begin{abstract}
In this short note we observe that the gamma factor defined by Gelfand and Kazhdan coincides with the Rankin-Selberg root number defined by Jacquet, Piatetskii-Shapiro and Shalika. 
\end{abstract}
\section{introduction}
In their seminal work \cite{MR701565} Jacquet, Piatetskii-Shapiro and Shalika defined the local invariants, $L$ and $\epsilon$ factors, for pairs of representations of general linear groups of arbitrary ranks over a non-archimedean local field. Earlier, Gelfand and Kazhdan introduced in \cite{MR0333080} a local gamma factor for a pair of representations $\pi$ cuspidal on $\GL(n)$ and $\tau$ generic on $\GL(n-1)$. Their construction is more direct and conceptual.  

To the best of my knowledge the connection between the two constructions was never documented, and I hope that doing so in these notes will be of service to the community. Recently, David Kazhdan asked me about the relation between the two constructions for $n=2$ and this led me to write this short note. I thank him for pointing the question to me and for his encouragement to post this note.  

\begin{remark}
For all facts mentioned bellow about representations of $\GL(n)$ and their Whittaker models we refer to \cite{MR0333080} and the references therein.
\end{remark}

\section{Notation}
Let $F$ be a non-archimedean local field and $\psi$ a non trivial additive character of $F$.
For $n\in \N$ let $G=G_n=\GL_n(F)$, $U=U_n$ its subgroup of upper-triangular unipotent matrices and $\theta=\theta_n$ the generic character of $U$ defined by
\[
\theta(u)=\psi(u_{1,2}+\cdots +u_{n-1,n}).
\]
Let $P$ be the mirabolic subgroup of matrices in $G$ with last raw equal to $(0,\dots,0,1)$. We write $P=V\rtimes G_{n-1}$ where $V\simeq F^{n-1}$ is the unipotent radical of $P$. We view $G_{n-1}$ as a subgroup of $P$ via $g\mapsto \diag(g,1)$.

Write $i_U^G(\theta)$ for smooth induction of compact support and $I_U^G(\theta)$ for smooth induction from $(U,\theta)$ to $G$. We similarly define $i_U^P(\theta)$ and $I_U^P(\theta)$. Thus for $Q\in \{P,G\}$ we have that $I_U^Q(\theta)$ is the contragredient (smooth dual) of $i_U^Q(\bar\theta)$ with pairing
\[
\sprod{f}{\varphi}=\int_{U\bs Q} f(q)\varphi(q)\ dq, \ \ \ f\in i_U^Q(\bar\theta),\ \varphi\in I_U^Q(\theta).
\]

The restriction to $P$ map $W\mapsto W|_{P}$ defines surjective $P$-equivariant morphisms
\[
i_U^G(\theta)\rightarrow i_U^P(\theta)\rightarrow 0\ \ \ \text{and}\ \ \ I_U^G(\theta)\rightarrow I_U^P(\theta)\rightarrow 0.
\]
We further point out that restriction to $G_{n-1}$ gives $G_{n-1}$-equivariant isomorphisms 
\[
I_U^P(\theta)\simeq I_{U_{n-1}}^{G_{n-1}}(\theta_{n-1})\ \ \ \text{and}\ \ \ i_U^P(\theta)\simeq i_{U_{n-1}}^{G_{n-1}}(\theta_{n-1}).
\]
The inverse map $f\mapsto \phi$ is given by
\[
\phi(vg)=\theta(v)f(g),\ \ \ v\in V,\ g\in G_{n-1}.
\]

By a representation, we always mean a smooth representation. We say that an irreducible representation $(\pi,\Nu)$ of $G$ is generic if $\Hom_U(\pi,\theta)\ne 0$. In that case, this space is one dimensional and fixing a non-zero $l\in \Hom_U(\pi,\theta)$ defines a map $\xi\mapsto W_\xi^{\pi,l}: \Nu\rightarrow I_U^G(\theta)$ by
\[
W_\xi^{\pi,l}(g)=l(\pi(g)\xi),\ \ \ g\in G.
\]
This map defines the unique up to scalar $G$-isomorphism of $\pi$ with a subspace of $I_U^G(\theta)$ and its image, the 
Whittaker model of $\pi$, is denoted by $\Wh(\pi,\theta)$. The Kirillov model of $\pi$, denoted $\K(\pi,\theta)$ is the image of the restriction to $P$ map $\Wh(\pi,\theta)\rightarrow \K(\pi,\theta)$. The restriction to $P$ map is a $P$-isomorphism $\Wh(\pi,\theta)\simeq \K(\pi,\theta)$.

If in addition $\pi$ is cuspidal then $\K(\pi,\theta)=i_U^P(\theta)$.

Let $\iota$ be the involution on $G$ defined by transpose inverse: $g^\iota={}^tg^{-1}$. For a representation $(\pi,\Nu)$ of $G$ let $(\pi^\iota,\Nu)$ be the representation $\pi\circ \iota$ of $G$ on $\Nu$. If $\pi$ is irreducible then $\pi^\iota\simeq \tilde\pi$ is the contragredient of $\pi$.
Let 
\[
w_n=(\delta_{i,n+1-j})\in G
\]
and for $W\in I_U^G(\theta)$ let $\wt W\in I_U^G(\bar\theta)$ be defined by $\wt W(g)=W(w_n g^\iota)$. For an irreducible, generic representation $\pi$ of $G$ we have that $W\mapsto \wt W$ maps $\Wh(\pi,\theta)$ to $\Wh(\pi^\iota,\bar\theta)$.

Finally, set 
\[
\epsilon_n=\diag((-1)^{n-1},\dots,-1,1)\in G
\] 
and for $W\in \Wh(\pi,\theta)$ let $W^{\epsilon}\in \Wh(\pi,\bar\theta)$ be defined by $W^\epsilon(g)=W(\epsilon_ng)$. The map $W\mapsto W^{\epsilon}:\Wh(\pi,\theta)\rightarrow\Wh(\pi,\bar\theta)$ is a $G$-isomorphism.
In particular, for $W\in \Wh(\pi,\theta)$ we have that $(\wt W)^\epsilon\in \Wh(\pi^\iota,\theta)=\Wh(\tilde\pi,\theta)$.

\section{The Jacquet, Piateskii-Shapiro and Shalika construction}
Let $\pi$ resp. $\tau$ be an irreducible generic representation of $G$ resp. $G_{n-1}$. 
The associated family of zeta integrals is defined in \cite{MR701565} by the convergent integral
\[
Z(s,W,W')=\int_{U_{n-1}\bs G_{n-1}}W(g)W'(g)\abs{\det g}^{s-\frac12}\ dg,\ \ \ W\in \Wh(\pi,\theta), W'\in \Wh(\tau,\bar\theta)
\]
for $\Re(s)\gg 1$ and by meromorphic continuation to a rational function of $q^{-s}$ for general $s$.
The gamma factor $\gamma(s,\pi,\tau,\psi)$ associated to the pair $\pi,\tau$ is characterized by the functional equation
\begin{equation}\label{eq JPSS}
Z(1-s,\wt W,\wt W')=\omega_\tau(-1)^{n-1}\gamma(s,\pi,\tau,\psi)Z(s,W,W')
\end{equation}
where $\omega_\tau$ is the central character of $\tau$.

If $\pi$ is cuspidal then the zeta integrals converge for every $s$ and can by made non-zero for some data $W, W'$ as above.

\section{The Gelfand-Kazhdan construction}
Let $s_n=((-1)^{i-1}\delta_{i,n+1-j})\in G$.
Since $s_n^2=(-1)^{n-1}I_n$ is central in $G$, the conjugation $\Ad(s_n)$ by $s_n$ is an involution on $G$. Furthermore, $s_n^\iota=s_n$ and consequently 
\[
\jmath:=\Ad(s_n)\circ \iota=\iota\circ \Ad(s_n)
\] 
is an involution on $G$. Note that $\jmath$ preserves $U$ and stabilizes $\theta$. Consequently, given a representation $(\pi,\Nu)$ of $G$ and $l\in \Hom_U(\pi,\theta)$ the linear form $\hat l=l\circ \pi(s_n)$ on $\Nu$ satisfies $\hat l\in \Hom_U(\pi^\iota,\theta)$.

Assume from now on that $\pi$ is irreducible and cuspidal. Then the map $\varphi_{\pi,l}:\Nu \rightarrow i_U^P(\theta)$ defined by $\varphi_{\pi,l}(\xi)=W_\xi^{\pi,\ell}|_{P}$ is a $P$-isomorphism. Furthermore, 
\[
K(\pi)=\varphi_{\pi^\iota,\hat l}\circ \varphi_{\pi,l}^{-1}:i_U^P(\theta)\rightarrow i_U^P(\theta)
\] 
is an isomorphism independent of $l$. In fact, $K(\pi)\in \Hom_{G_{n-1}}(i_U^P(\theta),i_U^P(\theta)^\iota)$.

We further define the operator $A$ on $i_U^P(\theta)$ by 
\[
(Af)(vg)=f(vs_{n-1}g^\iota),\ \ \ f\in i_U^P(\theta), v\in V,g\in G_{n-1}.
\]
Note that also $A\in \Hom_{G_{n-1}}(i_U^P(\theta),i_U^P(\theta)^\iota)=\Hom_{G_{n-1}}(i_U^P(\theta)^\iota,i_U^P(\theta))$ and therefore that $C(\pi)=A\circ K(\pi)\in\Hom_{G_{n-1}}(i_U^P(\theta),i_U^P(\theta))$.
We denote by $C^*(\pi)$ the adjoint operator on $I_U^P(\bar\theta)$ so that
\[
\sprod{C(\pi)f}{\varphi}=\sprod{f}{C^*(\pi)\varphi},\ \ \ f\in i_U^P(\theta),\ \varphi\in I_U^P(\bar \theta).
\]

Let $\tau$ be an irreducible, generic representation of $G_{n-1}$ and identify its Whittaker model $\Wh(\tau,\bar\theta)\subseteq I_{U_{n-1}}^{G_{n-1}}(\bar\theta_{n-1})\simeq I_U^P(\bar\theta)$ as a subspace of $I_U^P(\bar\theta)$.
Then $C^*(\pi)$ restricts to a $G_{n-1}$-equivariant operator on $\Wh(\tau,\bar\theta)$ and consequently is a scalar operator. Gelfand and Kazhdan defined the gamma factor associated to $\pi$ and $\tau$ to be this scalar and we denote it by $\gamma_{\gk}(\pi,\tau,\psi)$.
Explicitly, $\sprod{f}{C^*(\pi)W'}=\gamma_{\gk}(\pi,\tau,\psi)\sprod{f}{W'}$ or as we prefer to write it
\begin{equation}\label{eq gkfe}
\sprod{K(\pi)f}{A^*W'}=\gamma_{\gk}(\pi,\tau,\psi)\sprod{f}{W'},\ \ \ f\in i_U^P(\theta),\ W'\in \Wh(\tau,\bar\theta).
\end{equation}
\section{The comparison}
The main result of this note compares between the two constructions.
\begin{theorem}
Let $\pi$ be an irreducible cuspidal representation of $\GL_n(F)$ and $\tau$ an irreducible, generic representation of $\GL_{n-1}(F)$. Then 
\[
\gamma_{\gk}(\pi,\tau,\psi) = \gamma(\frac12,\pi,\tau,\psi).
\]
\end{theorem}
\begin{proof}
Note that for $\xi\in\Nu$, the space of $\pi$,  we have
\[
K(\pi)(W_\xi^{\pi,l}|_{P})=W_\xi^{\pi^\iota,\hat l}|_{P}.
\]
Since $s_n=w_n\epsilon_n$ and $\epsilon_n^\iota=\epsilon_n$, for $W=W_\xi^{\pi,l}\in \Wh(\pi,\Theta)$ we have 
\[
W_\xi^{\pi^\iota,\hat l}(g)=W(s_ng^\iota)=(\wt W)^{\epsilon}( g).
\]
Next we explicate $A^*W'(g)$ for $g\in G_{n-1}$. For $f\in i_U^P(\theta)$ and $\varphi\in I_U^P(\bar\theta)$ we have that
\[
\sprod{f}{A^*\varphi}=\sprod{Af}{\varphi}=\int_{U_{n-1}\bs G_{n-1}}f(s_{n-1}g^\iota)\varphi(g)\ dg=\int_{U_{n-1}\bs G_{n-1}}f(g)\varphi(s_{n-1}^{-1} g^\iota)\ dg
\]
and consequently 
\[
(A^*W')(g)=W'(s_{n-1}^{-1}g^\iota),\ \ \ g\in G_{n-1}.
\]
Note that $s_{n-1}^{-1}=(-1)^n w_{n-1}\epsilon_{n-1}$ so that for $f=W|_{P}$ we have
\[
\sprod{K(\pi)f}{A^*W'}=\omega_\tau(-1)^n\int_{U_{n-1}\bs G_{n-1}}\wt W(\epsilon_n g)\wt W'(\epsilon_{n-1}g)\ dg.
\]
Note that $\epsilon_n\in G_{n-1}$ so that the variable change $g\mapsto \epsilon_n g$ makes sense on $U_{n-1}\bs G_{n-1}$. Furthermore, $\epsilon_n\epsilon_{n-1}=-I_{n-1}$ and we conclude that 
\[
\sprod{K(\pi)f}{A^*W'}=\omega_\tau(-1)^{n-1}\int_{U_{n-1}\bs G_{n-1}}\wt W(g)\wt W'(g)\ dg.
\]
To conclude, \eqref{eq gkfe} now reads
\[
Z(\frac12,\wt W, \wt W')=\omega_\tau(-1)^{n-1}\gamma_{\gk}(\pi,\tau,\psi) Z(\frac12,W,W').
\]
Comparing with \eqref{eq JPSS} and in light of the paragraph below it, the theorem follows.

\end{proof}

\def\cprime{$'$} \def\Dbar{\leavevmode\lower.6ex\hbox to 0pt{\hskip-.23ex
  \accent"16\hss}D} \def\cftil#1{\ifmmode\setbox7\hbox{$\accent"5E#1$}\else
  \setbox7\hbox{\accent"5E#1}\penalty 10000\relax\fi\raise 1\ht7
  \hbox{\lower1.15ex\hbox to 1\wd7{\hss\accent"7E\hss}}\penalty 10000
  \hskip-1\wd7\penalty 10000\box7}
  \def\polhk#1{\setbox0=\hbox{#1}{\ooalign{\hidewidth
  \lower1.5ex\hbox{`}\hidewidth\crcr\unhbox0}}} \def\dbar{\leavevmode\hbox to
  0pt{\hskip.2ex \accent"16\hss}d}
  \def\cfac#1{\ifmmode\setbox7\hbox{$\accent"5E#1$}\else
  \setbox7\hbox{\accent"5E#1}\penalty 10000\relax\fi\raise 1\ht7
  \hbox{\lower1.15ex\hbox to 1\wd7{\hss\accent"13\hss}}\penalty 10000
  \hskip-1\wd7\penalty 10000\box7}
  \def\ocirc#1{\ifmmode\setbox0=\hbox{$#1$}\dimen0=\ht0 \advance\dimen0
  by1pt\rlap{\hbox to\wd0{\hss\raise\dimen0
  \hbox{\hskip.2em$\scriptscriptstyle\circ$}\hss}}#1\else {\accent"17 #1}\fi}
  \def\bud{$''$} \def\cfudot#1{\ifmmode\setbox7\hbox{$\accent"5E#1$}\else
  \setbox7\hbox{\accent"5E#1}\penalty 10000\relax\fi\raise 1\ht7
  \hbox{\raise.1ex\hbox to 1\wd7{\hss.\hss}}\penalty 10000 \hskip-1\wd7\penalty
  10000\box7} \def\lfhook#1{\setbox0=\hbox{#1}{\ooalign{\hidewidth
  \lower1.5ex\hbox{'}\hidewidth\crcr\unhbox0}}}
\providecommand{\bysame}{\leavevmode\hbox to3em{\hrulefill}\thinspace}
\providecommand{\MR}{\relax\ifhmode\unskip\space\fi MR }
\providecommand{\MRhref}[2]{%
  \href{http://www.ams.org/mathscinet-getitem?mr=#1}{#2}
}
\providecommand{\href}[2]{#2}


\end{document}